\documentclass[a4paper,12pt]{article}
\usepackage{latexsym,amsmath,amsthm,amssymb}
\usepackage{a4wide}
\usepackage{hyperref}
\usepackage{marginnote}
\usepackage{color}
\usepackage{cite}
\hypersetup{
pdftitle={Second order Sobolev Hn}   
pdfauthor={Van Hoang Nguyen},
colorlinks = true,
linkcolor = magenta,
citecolor = blue,
}

\theoremstyle{plain}
\newtheorem{theorem}{Theorem}[section]

\newtheorem{proposition}[theorem]{Proposition}
\newtheorem{lemma}[theorem]{Lemma}
\newtheorem{corollary}[theorem]{Corollary}

\theoremstyle{definition}

\theoremstyle{remark}

\renewcommand{\thefootnote}{\arabic{footnote}}

\def\R{\mathbb R}
\def\N{\mathbb N}

\def\Om{\Omega}
\def\De{\Delta} 

\def\si{\sigma}
\def\lam{\lambda}
\def\ep{\epsilon}
\def\na{\nabla}
\def\pa{\partial}
\def\la{\langle} 
\def\ra{\rangle} 
\def\lt{\left}
\def\rt{\right}

\def\H{\mathbb H}
\def\B{\mathbb B}

\numberwithin{equation}{section}

\title{Second order Sobolev type inequalities in the hyperbolic spaces}
\author{Van Hoang Nguyen\footnote{Institute of Research and Development, Duy Tan University, Da Nang, Vietnam.}
}

\begin{document}
\maketitle


\renewcommand{\thefootnote}{}

\footnote{Email: \href{mailto: Van Hoang Nguyen <vanhoang0610@yahoo.com>}{vanhoang0610@yahoo.com}.}

\footnote{2010 \emph{Mathematics Subject Classification\text}: 26D10, 46E35.}

\footnote{\emph{Key words and phrases\text}: Poincar\'e--Sobolev inequality, Sobolev inequality, Adams inequality, Adams inequality with exact growth, Hyperbolic spaces, decreasing spherical symmetric rearrangement.}

\renewcommand{\thefootnote}{\arabic{footnote}}
\setcounter{footnote}{0}

\begin{abstract}
We establish several Poincar\'e--Sobolev type inequalities for the Lapalce--Beltrami operator $\Delta_g$ in the hyperbolic space $\H^n$ with $n\geq 5$. These inequalities could be seen as the improved second order Poincar\'e inequality with remainder terms involving with the sharp Rellich inequality or sharp Sobolev inequality in $\H^n$. The novelty of these inequalities is that it combines both the sharp Poincar\'e inequality and the sharp Rellich inequality or the sharp Sobolev inequality for $\Delta_g$ in $\H^n$. As a consequence, we obtain the Poincar\'e--Sobolev inequality for the second order GJMS operator $P_2$ in $\H^n$. In dimension $4$, we obtain an improvement of the sharp Adams inequality and an Adams inequality with exact growth for radial functions in $\H^4$.
\end{abstract}

\section{Introduction}
Let $\H^n$ denote the hyperbolic space of dimension $n$ with $n \geq 2$. In this paper, we shall use the Poincar\'e ball model for the hyperbolic spaces.  This is the unit ball
\[
\B^n =\{x = (x_1,\ldots,x_n) \in \R^n\, :\, |x| =(x_1^2 + \cdots + x_n^2)^{\frac12} < 1\},
\]
equipped with the usual Poincar\'e metric
\[
g(x) = \lt(\frac 2{1 -|x|^2}\rt)^2 \sum_{i=1}^n d x_i^2.
\]
The hyperbolic volume element with respect to $g$ is $dV_g = (\frac 2{1 -|x|^2})^n dx$. Let $d$ denote the geodesic distance with respect to $g$ in $\H^n$, then the distance from the origin to $x$ is $\rho(x):=d(0,x) = \ln \frac{1+ |x|}{1 -|x|}$. Let $\nabla_g$ denote the hyperbolic gradient with respect to $g$ in $\H^n$. We have $\na_g = (\frac{1 -|x|^2}2)^2 \na$, where $\na$ is the usual gradient in the Euclidean space $\R^n$. The Laplace-Beltrami operator with respect to $g$ in $\H^n$ is given by
\[
\Delta_g = \lt(\frac{1 -|x|^2}{2}\rt)^2 \Delta + (n-2) \frac{1 -|x|^2}2 \la x, \na \ra,
\]
where $\Delta$ and $\la \cdot, \cdot \ra$ denote the usual Laplace operator and the usual scalar product in $\R^n$ respectively. The spectral gap of $-\De_g$ on $L^2(\H^n)$ is $\frac{(n-1)^2}4$ (see, e.g., \cite{MS,KS}), i.e., 
\begin{equation}\label{eq:Poincare1}
\int_{\H^n} |\na_g u|_g^2 dV_g \geq \frac{(n-1)^2}4 \int_{\H^n} u^2 dV_g,\quad u \in C_0^\infty(\H^n),
\end{equation}
where we use $|\na_g u|_g = \sqrt{g(\na_g u, \na_g u)}$ for simplifying notation. The inequality \eqref{eq:Poincare1} is sharp and the sharp constant $\frac{(n-1)^2}4$ is never attained. The non-attainability of \eqref{eq:Poincare1} together with the sub-criticality of the operator $-\Delta_g  - \frac{(n-1)^2}4$ leaves a room for several improvements of the inequality \eqref{eq:Poincare1}. For examples, the reader could consult the paper \cite{BDGG,BGG,BG,BGGP} for the improvements of \eqref{eq:Poincare1} by adding the remainder terms concerning to the Hardy weight, i.e., 
\[
\int_{\H^n} |\na_g u|_g^2 dV_g -\frac{(n-1)^2}4 \int_{\H^n} u^2 dV_g \geq C \int_{\H^n} W u^2 dV_g,
\]
for some constant $C >0$ and the weight $W$ satisfying some appropriate conditions. Another improvement of \eqref{eq:Poincare1} was obtained by Mancini and Sandeep \cite{MS} for the dimension $n \geq 3$,
\[
\int_{\H^n} |\na_g u|_g^2 dV_g -\frac{(n-1)^2}4 \int_{\H^n} u^2 dV_g \geq C  \lt(\int_{\H^n} |u|^q dV_g\rt)^{\frac 2q},
\]
where $2 < q \leq \frac{2n}{n-2}$ and $C$ is a positive constant. The preceding inequality is equivalent to the Hardy--Sobolev--Maz'ya inequality in the half space (see \cite[Section $2.1.6$]{Maz'ya}). In the critical case $q = \frac{2n}{n-2}$, we have
\begin{equation}\label{eq:HSM1}
\int_{\H^n} |\na_g u|_g^2 dV_g - \frac{(n-1)^2}4 \int_{\H^n} u^2 dV_g \geq C_n\lt(\int_{\H^n} |u|^{\frac{2n}{n-2}} dV_g\rt)^{\frac{n-2}n},\quad u\in C_0^\infty(\H^n),
\end{equation}
where $C_n$ denotes the sharp constant for which \eqref{eq:HSM1} holds true. It was proved by Tertikas and Tintarev \cite{TT} for $n \geq 4$ that $C_n < S_1(n,2)$ where $S_1(n,2)$ is the sharp constant in the first order $L^2$-Sobolev inequality in $\R^n$ (see \cite{Aubin,Talenti}) and $C_n$ is attained. In the case $n=3$, it was shown by Benguria, Frank and Loss \cite{BFL} that $C_3 = S_1(3,2)$ and is not attained. Another proof of the non-attainability of $C_3$ was given by Mancini and Sandeep \cite{MS}. For the $L^p-$version of \eqref{eq:HSM1}, we refer the reader to \cite{VHN2018a}. In that paper, the author proved the following inequality
\begin{equation}\label{eq:Nguyen}
\int_{\R^n} |\na_g u|_g^p dV_g - \frac{(n-1)^p}{p^p} \int_{\H^n} |u|^p dV_g \geq S_1(n,p) \lt(\int_{\H^n} |u|^{\frac{np}{n-p}} dV_g\rt)^{\frac{n-p}n},\, u\in C_0^\infty(\H^n),
\end{equation}
for $n\geq 4$ and $\frac{2n}{n-1}\leq p < n$, where $S_1(n,p)$ is the best constant in the first order $L^p-$Sobolev inequality in $\R^n$ (see \cite{Aubin,Talenti}). Notice that $(\frac{n-1}p)^p$ is the spectral gap of the $p-$Laplacian in the hyperbolic space (see \cite{NgoNguyen2017}).

In this paper, we extend the inequality \eqref{eq:HSM1} to the second order derivative (i.e., the Laplace-Beltrami operator $\Delta_g$ with respect to $g$ in $\H^n$). Let us recall the following second order Poincar\'e inequality in $\H^n$ (see \cite{KS,NgoNguyen2017})
\begin{equation}\label{eq:Poincare2}
\int_{\H^n} (\Delta_g u)^2 dV_g \geq \frac{(n-1)^4}{16} \int_{\H^n} u^2 dV_g, \quad u\in C_0^\infty(\H^n).
\end{equation}
The constant $\frac{(n-1)^4}{16}$ is sharp and is never attained. The first main result of this paper is the following improvement of \eqref{eq:Poincare2} with the remainder involving with the Rellich type inequality. Let us recall the sharp Rellich inequality in $\H^n$ for $n \geq 5$ (see \cite{KO})
\begin{equation}\label{eq:RellichHn}
\int_{\H^n} (\Delta_g u)^2 dV_g \geq \frac{n^2(n-4)^2}{16} \int_{\H^n} \frac{u(x)^2}{\rho(x)^4} dV_g,\quad u\in C_0^\infty(\H^n),
\end{equation}
with the constant $\frac{n^2(n-4)^2}{16}$ being sharp. We shall prove the following theorem.

\begin{theorem}\label{Rellich}
Let $n \geq 5$. Then the following inequality holds
\begin{equation}\label{eq:improveRellich}
\int_{\H^n} (\Delta_g u)^2 dV_g - \frac{(n-1)^4}{16} \int_{\H^n} u^2 dV_g \geq \frac{n^2(n-4)^2}{16} \int_{\H^n} \frac{u(x)^2}{(\frac{V_g(B_g(0,\rho(x)))}{\si_n})^{\frac4n}} dV_g,
\end{equation}
for any $u\in C_0^\infty(\H^n)$, where $B_g(0,r)$ denotes the geodesic ball with center at $0$ and radius $r >0$ in $\H^n$ and $\si_n$ is the volume of unit ball in $\R^n$. Morover, the constant $\frac{n^2(n-4)^2}{16}$ in the right hand side of \eqref{eq:improveRellich} is sharp.
\end{theorem}
It is easy to check that the weight $W(x) =(\frac{V_g(B_g(0,\rho(x)))}{\si_n})^{-\frac4n}$ satisfies the following asymptotic estimates $W(x) \sim \rho(x)^{-4}$ as $x \to 0$ and
\[
W(x) \sim \lt(\frac{n}{n-1}\rt)^{-\frac 4n} (\sinh \rho(x))^{-\frac{4(n-1)}n}\quad \text{as}\quad |x| \to 1.
\]
Thus in a neighborhood of the origin, the right hand side of \eqref{eq:improveRellich} is similar with the one of \eqref{eq:RellichHn}. However, we can chek that $W(x) < \rho(x)^{-4}$ for any $x\not=0$. Hence, it is interesting problem whether we can replace the right hand side of \eqref{eq:improveRellich} by the one of \eqref{eq:RellichHn}.

The second main result of this paper is an improvement of \eqref{eq:Poincare2} with the remainder involving with the sharp Sobolev inequality in $\H^n$. Let $S_2(n,2)$ denote the sharp constant in the second order $L^2-$Sobolev inequality in $\R^n$ with $n\geq 5$, i.e.,
\begin{equation}\label{eq:Sobolevorder2}
S_2(n,2)\lt(\int_{\R^n} |u|^{\frac{2n}{n-4}} dx\rt)^{\frac{n-4}n} \leq \int_{\R^n} (\De u)^2 dx.
\end{equation}
The sharp constant $S_2(n,2)$ was computed by Lieb \cite{Lieb} by proving the sharp Hardy--Littlewood--Sobolev inequality which is the dual version of \eqref{eq:Sobolevorder2}. The following theorem gives an extension of \eqref{eq:HSM1} for the Laplace--Beltrami opertor $\Delta_g$ on $\H^n$. It is interesting that we reach the sharp constant $S_2(n,2)$ in \eqref{eq:Sobolevorder2} in the right hand side of obtained inequality. 

\begin{theorem}\label{Sobolev}
Let $n\geq 5$. Then the following inequality holds
\begin{equation}\label{eq:improveSobolev}
\int_{\H^n} (\Delta_g u)^2 dV_g - \frac{(n-1)^4}{16} \int_{\H^n} u^2 dV_g \geq S_2(n,2) \lt(\int_{\H^n} |u|^{\frac{2n}{n-4}} dV_g\rt)^{\frac{n-4}n},
\end{equation}
for any function $u\in C_c^\infty(\H^n)$. The constant $S_2(n,2)$ in the right hand side of \eqref{eq:improveSobolev} is sharp.
\end{theorem}
In recent paper \cite{LuYanga}, Lu and Yang extended the inequality \eqref{eq:HSM1} to higher order derivatives. Their results was established for the GJMS operator $P_k$, $k\geq 1$ in $\H^n$ (see \cite{GJMS,Juhl})
\[
P_k = P_1(P_1+2) \cdots (P_1 +k(k-1)), \quad k \in \N,
\]
where $P_1 = -\De_g - \frac{n(n-2)}4$ is the conformal Laplacian in $\H^n$. The sharp Sobolev inequality for $P_k$ (see \cite{Liu} and \cite{LuYanga} for another proof) reads as
\begin{equation}\label{eq:SobolevPk}
\int_{\H^n} (P_ku) u dV_g \geq S_k(n,2) \lt(\int_{\H^n} |u|^{\frac{2n}{n-2k}} dV_g\rt)^{\frac{n-2k}n}, \, u\in C_0^\infty(\H^n),\, 1 \leq k < \frac n2,
\end{equation}
where $S_k(n,2)$ is the sharp constant in the $k-$th order Sobolev inequality in $\R^n$. In other hand, we have the following Poincar\'e inequality for $P_k$
\begin{equation}\label{eq:PoincarePk}
\int_{\H^n} (P_ku) u dV_g \geq \prod_{i=1}^k \frac{(2i-1)^2}4 \int_{\H^n} u^2 dV_g,\quad \quad u\in C_0^\infty(\H^n).
\end{equation}
The following inequality was proved by Lu and Yang for $n > 2k$ and $u\in C_0^\infty(\H^n)$
\begin{equation}\label{eq:LuYangineq}
\int_{\H^n} (P_ku) u dV_g - \prod_{i=1}^k \frac{(2i-1)^2}4 \int_{\H^n} u^2 dV_g \geq C_n \lt(\int_{\H^n} |u|^{\frac{2n}{n-2k}} dV_g\rt)^{\frac{n-2k}n},
\end{equation}
for some constant $C_n >0$. The proof of Lu and Yang is based on the sharp Hardy--Littlewood--Sobolev inequality in $\H^n$ which is different with the one in this paper. Moreover the constant $C_n$ is not explicit. As a consequence of our Theorem \ref{Sobolev}, we obtain the inequality of Lu and Yang for $k=2$ with an explicit constant $\frac{5S_2(n,2)}{(n-1)^2}$ (see Corollary \ref{cor} below).

In the dimension four, i.e., $n =4$, we have the following Adams type inequalities

\begin{theorem}\label{Adams}
For any $\lam < \frac{81}{16}$, then there exists a constant $C_\lam >0$ such that
\begin{equation}\label{eq:Adams}
\sup_{\int_{\H^4} |\De_g u|^2 dV_g - \lam \int_{\H^4} u^2 dV_g \leq 1} \int_{\H^4} \lt(e^{32\pi^2 |u|^2} -1\rt) dV_g \leq C_\lam.
\end{equation}
There exists a constant $C >0$ such that
\begin{equation}\label{eq:Adamsexact}
\sup_{\substack{\int_{\H^4} |\De_g u|^2 dV_g - \frac{81}{16} \int_{\H^4} u^2 dV_g \leq 1,\\ u \, \, \text{\rm is radial}}} \frac1{\int_{\H^4} u^2 dV_g} \int_{\H^4} \frac{e^{32\pi^2 u^2} -1}{(1+ |u|)^2}dV_g \leq C.
\end{equation}
Moreover, the power $2$ in the denominator is sharp in the sense that the supremum will be infinite if we replace it by any power $p < 2$.
\end{theorem}
Adams inequality is the limit case of the Sobolev embedding. The sharp Adams inequality was firstly established for bounded domains in the Euclidean space by Adams \cite{Adams}. Adams inequality is the higher order version of the well known Moser--Trudinger inequality involving with the first order derivative (see \cite{M1970,P1965,T1967,Y1961}). The Adams inequality was extended to unbounded domains and whole space by Ruf and Sani \cite{RufSani}, Lam and Lu \cite{LamLu} and Fontana and Morpurgo \cite{FM}. Especially, we have the following inequality in $\R^4$,
\begin{equation}\label{eq:AdamsR4}
\sup_{\int_{\R^4} (\De u)^2 dx + \int_{\R^4} u^2 dx \leq 1} \int_{\R^4} \lt(e^{32\pi^2 u^2} -1\rt) dx < \infty.
\end{equation}
The Adams inequality was established in the hyperbolic space by Karmakar and Sandeep \cite{KS}. Their results in $\H^4$ is the following
\begin{equation}\label{eq:KarSan}
\sup_{\int_{\H^4} (P_2 u) u dV_g \leq 1} \int_{\H^4} \lt(e^{32\pi^2 u^2} -1\rt) dV_g < \infty.
\end{equation}
Notice that the condition $\int_{\H^4} (P_2 u) u dV_g \leq 1$ and ours in \eqref{eq:Adams} can not compare in general. We refer the readers to the paper of Lu and Yang \cite{LuYang} for several Hardy--Adams inequality in $\H^4$ which extends the Hardy--Moser--Trudinger inequality of Wang and Ye \cite{WY} to second order derivative. The inequality \eqref{eq:Adams} can be derived from the results in \cite{LuYang}, however our approach here is different with theirs by using the technique from Fourier analysis in the hyperbolic spaces and Adams' lemma in \cite{Adams}. The inequality \eqref{eq:Adamsexact} is a kind of Adams inequality with exact growth in $\H^4$. However, it is stated only for radial functions. In \cite{Karmakar}, an Adams inequality with exact growth in $\H^4$ was proved for any function $u$ on $\H^4$ under the condition $\int_{\H^4} (P_2 u) u dV_g \leq 1$. Notice that $\int_{\H^4} (P_2 u) u dV_g$ is equivalent with a full Sobolev norm in $\H^4$ which is different with our functional $\int_{\H^4} (\Delta_g u)^2 dV_g -\frac{81}{16} \int_{\H^4} u^2 dV_g$. So it is interesting if we could prove \eqref{eq:Adamsexact} for any funtions on $\H^4$. The Adams inequality with exact growth in $\R^4$ was proved by Masmoudi and Sani \cite{MS} in the following form
\begin{equation}\label{eq:MasmoudiSani}
\sup_{\int_{\R^4} (\De u)^2 dx \leq 1 } \frac{1}{\int_{\R^4} u^2 dx} \int_{\R^4} \frac{e^{32 \pi^2 u^2} -1}{(1 + |u|)^2} dx < \infty.
\end{equation}
The inequality \eqref{eq:MasmoudiSani} implies \eqref{eq:AdamsR4} and plays important role in the proof of \eqref{eq:Adamsexact}. It was extended to any dimension $n\geq 3$ in \cite{LTZ}. We refer the readers to \cite{MS2017} for the Adams inequality with exact growth for higher order derivatives. 

We conclude this introduction by giving some comments on the proof of our main theorems. Our proof is based on the decreasing spherical symmetric rearrangement applied to the setting of hyperbolic spaces. By this technique, we reduce the proof of our theorems to the radial functions in $\H^n$. The next step is the novelty of our approach. Given a radial function $u\in W^{2,2}(\H^n)$, let $v$ be the function on $[0,\infty)$ such that $u(x) = v(V_g(B_g(0,\rho(x))))$, $x\in \H^n$. Using this function $v$, we define a function $u_e$ in $\R^n$ by $u_e(y) = v(\si_n |y|^n)$ with $y\in \R^n$. By the definition of $u_e$, we can easily check that
\begin{equation}\label{eq:dangthucnorm}
\int_{\H^n} \Phi(u) dV_g = \int_0^\infty \Phi(v(s)) ds = \int_{\R^n} \Phi(u_e(y)) dy,
\end{equation}
for any function $\Phi$ on $\R$ such that one of the integrals above is well defined. Our way is to establish a deviation between $\int_{\H^n} (\Delta_g u)^2 dV_g$ and $\int_{\R^n} (\Delta u_e)^2 dy$. The following result is the key tool in the proof of our main theorems.
\begin{theorem}\label{keytool}
Let $n \geq 4$ and let $u\in W^{2,2}(\H^n)$ be a radial function. Suppose $u(x) = v(V_g(B_g(0,\rho(x))))$ for some function $v$ on $[0,\infty)$, and define the function $u_e$ in $\R^n$ by $u_e(y) = v(\si_n |y|^n)$. Then there holds
\begin{equation}\label{eq:keytool}
\int_{\H^n} (\Delta_g u)^2 dV_g -\frac{(n-1)^4}{16} \int_{\H^n} u^2 dV_g \geq \int_{\R^n} (\De u_e)^2 dy.
\end{equation}
\end{theorem}
Our main Theorems \ref{Rellich}, \ref{Sobolev} and \eqref{eq:Adams} follow from Theorem \ref{keytool}, the equality \eqref{eq:dangthucnorm} and the similar inequalities in the Euclidean space $\R^n$. This approach was already used by the author in \cite{VHN2018,VHN2018a} to establish several improvements of the Moser--Trudinger inequality and Poincar\'e--Sobolev in the hyperbolic spaces for the first order derivative (e.g., the inequality \eqref{eq:Nguyen} above).

The organization of this paper is as follows. In next Section $2$, we recall some facts about the hyperbolic spaces and the decreasing spherical symmetric rearrangement technique in the hyperbolic spaces which enable us reduce the proof of the main theorems to the radial functions. Section $3$ is devoted to prove Theorem \ref{keytool} which is the key tool in our proof. Theorems \ref{Rellich}, \ref{Sobolev}, \ref{Adams} will be proved in Section $4$. 
 
\section{Preliminaries}
In this section we shall list some properties of the hyperbolic spaces and the symmetrization technique in the hyperbolic spaces. Let $\H^n$ denote the hyperbolic space of dimension $n\geq 2$. The hyperbolic space is a complete and simply connected Riemannian manifold having constant sectional curvature equal to $-1$. There are several models for $\H^n$, for examples, the half space model, the Poincar\'e ball model, the hyperboloid or Lorentz model, etc. All these models are equaivalent. As mentioned in the introduction, we shall use throughout in this paper the Poincar\'e ball model for the hyperbolic space which is very useful for problems involving the rotational symmetry. It is known that the decreasing spherical symmetric rearrangement technique works well in the setting of the hyperbolic space (see, e.g, \cite{Ba}). It is not only the key tool to establish several important inequalities in the hyperbolic spaces such as first order Sobolev inequality, sharp Moser--Trudinger inequality, etc.... but also plays the crucial role in proving the results in this paper. Let us recall some facts on this technique. A measurable function $u:\H^n \to \R$ is called vanishing at the infinity if for any $t >0$ the set $\{|u| > t\}$ has finite $V_g-$measure, i.e.,
\[
V_g(\{|u|> t\}) = \int_{\{|u|> t\}} dV_g < \infty.
\]
For such a function $u$, its distribution function is defined by
\[
\mu_u(t) = V_g( \{|u|> t\}).
\]
Notice that $t \mapsto \mu_u(t)$ is non-increasing and right continuous on $(0,\infty)$. The decreasing rearrangement function $u^*$ of $u$ is defined by
\[
u^*(t) = \sup\{s > 0\, :\, \mu_u(s) > t\}.
\] 
The decreasing, spherical symmetry, rearrangement function $u^\sharp$ of $u$ is defined by
\[
u^\sharp(x) = u^*(V_g(B(0,d(0,x)))),\quad x \in \H^n.
\]
It is well-known that $u$ and $u_g^\sharp$ have the same decreasing rearrangement function (which is $u^*$). As a consequence of this equi-distribution, we have
\begin{equation}\label{eq:dongphanbo}
\int_{\H^n} \Phi(|u|) dV_g = \int_{\H^n} \Phi(u^\sharp) dV_g = \int_0^\infty \Phi(u^*(t)) dt,
\end{equation}
for any function $\Phi: [0,\infty) \to [0,\infty)$ which is increasing, continuous and $\Phi(0) =0$. Since $u^*$ is non-increasing function, the maximal function $u^{**}$ of $u^*$ is defined by
\[
u^{**}(t) = \frac1t \int_0^t u^*(s) ds.
\]
Evidently, $u^*(t) \leq u^{**}(t)$.  Moreover, as a consequence of Hardy inequality, we have
\begin{proposition}\label{hardy}
Let $u \in L^{p}(\H^n)$ with $1 < p < \infty$, then 
\[
\lt(\int_0^\infty u^{**}(t)^p dt\rt)^{\frac1p} \leq \frac p{p-1} \lt(\int_0^\infty u^*(t)^q dt\rt)^{\frac1q} = \frac p{p-1} \|u\|_{L^p(\H^n)}.
\]
\end{proposition}
For $1 < p < \infty$, let $W^{1,2}(\H^n)$ denote the first order $L^p-$Sobolev space in $\H^n$ which is the completion of $C_0^\infty(\H^n)$ under the norm $\|u\|_{W^{1,p}(\H^n)} = \lt(\int_{\H^n} |\na_g u|_g^p dV_g\rt)^{\frac1p}$. The P\'olya--Szeg\"o principle in $\H^n$ asserts that if $u\in W^{1,p}(\H^n)$ then $u^\sharp \in W^{1,p}(\H^n)$ and $\|u^\sharp\|_{W^{1,p}(\H^n)} \leq \|u\|_{W^{1,p}(\H^n)}$. Furthermore, the following Poincar\'e--Sobolev inequality holds (see \cite{KS} for $p =2$ and \cite{NgoNguyen2017} for general $p\in (1,\infty)$)
\[
\int_{\H^n} |\na_g u|_g^p dV_g \geq \lt(\frac{n-1}p\rt)^p \int_{\H^n} |u|^p dV_g.
\]
Similarly, let $W^{2,p}(\H^n)$ denote the second order $L^p-$Sobolev space in $\H^n$ which is the completion of $C_0^\infty(\H^n)$ under the norm $\|u\|_{W^{2,p}(\H^n)} = \lt(\int_{\H^n} |\Delta_g u|^p dV_g\rt)^{\frac12}$. It was proved by Ngo and the author in \cite{NgoNguyen2017}  (see also \cite{KS} for $p=2$) that
\[
\int_{\H^n} |\De_g u|^p dV_g \geq \lt(\frac{(n-1)^2}{p p'}\rt)^p \int_{\H^n} |u|^p dV_g, \quad p' = \frac{p}{p-1}.
\]
Moreover, we have
\[
\int_{\H^n} |\Delta_g u|^p dV_g \geq \lt(\frac{n-1}{p'}\rt)^p \int_{\H^n} |\na_g u|_g^p dV_g.
\]
Notice that unlike in the space $W^{1,p}(\H^n)$, we do not have an analogue of P\'olya--Szeg\"o principle in $W^{2,p}(\H^n)$. Hence, in order to prove your main results we will establish the following Talenti type comparison principle in the hyperbolic space $\H^n$ instead of the P\'olya--Szeg\"o principle. We refer the readers to \cite{Talenti} for the Talenti comparison principle in the Euclidean space $\R^n$. We next consider the case $p =2$. Suppose that $u\in W^{2,2}(\H^n)$, we set $f = -\Delta_g u \in L^2(\H^n)$. Then we have from Proposition \ref{hardy} that
\[
\int_0^\infty f^{**}(t)^2 dt \leq 4 \int_0^\infty f^*(t)^2 = 4 \|\Delta_g u\|_{L^2(\H^n)}^2.
\]
Obviously, $u \in L^2(\H^n)$ by the Poincar\'e--Sobolev inequality. For each $t >0$ denote $F_n(t)$ the unique solution of the equation
\begin{equation}\label{eq:Fnfunction}
n \si_n \int_0^{F_n(t)} \sinh^{n-1}(s) ds =t.
\end{equation}
The function $t\mapsto F_n(t)$ is strictly increasing and smooth on $(0,\infty)$ and $F_n(0) =0$. It was proved by Ngo and the author in \cite{NgoNguyen2016} that
\begin{equation}\label{eq:NgoNguyen}
u^*(t) \leq v(t):= \int_t^\infty \frac{s f^{**}(s)}{(n\si_n \sinh(F_n(s))^{n-1})^2} ds, \quad s >0.
\end{equation} 
Note that $\sinh(F_n(s))^{n-1} \sim \frac{n-1}{n\si_n} s$ when $s \to \infty$ then the integral in the right hand side of \eqref{eq:NgoNguyen} is well defined for any $t >0$. For $x\in \H^n$, we set $\bar u(x) = v(V_g(B_g(0,\rho(x))))$. It is easy to check that $-\Delta_g \bar u(x) = f^\sharp(x)$ for $x \in \H^n$ and $u^\sharp(x) \leq \bar u(x)$. Consequently, we have
\[
\int_{\H^n} (\Delta_g u)^2 dV_g - \frac{(n-1)^4}{16} \int_{\H^n} u^2 dV_g \geq \int_{\H^n} (\Delta_g \bar u)^2 dV_g - \frac{(n-1)^4}{16} \int_{\H^n} \bar u^2 dV_g,
\]
and by Hardy--Littlewood inequality
\[
\int_{\H^n} \frac{u(x)^2}{V_g(B_g(0,\rho(x)))^a} dV_g \leq \int_{\H^n} \frac{u^\sharp(x)^2}{V_g(B_g(0,\rho(x)))^a} dV_g \leq \int_{\H^n} \frac{\bar u(x)^2}{V_g(B_g(0,\rho(x)))^a} dV_g,
\]
for any $a \geq 0$, and 
\[
\int_{\H^n} \lt(e^{32\pi^2 u^2} -1\rt) dV_g \leq \int_{\H^n}\lt(e^{32 \pi^2 \bar u^2} -1\rt) dV_g.
\]
Hence, it is enough to prove our main Theorems \ref{Rellich}, \ref{Sobolev} and \ref{Adams} for funtion $u$ which is nonnegative, decreasing, radially symmetric function such that $-\De_g u$ is nonnegative, decreasing, radially symmetric function too. Especially, it suffices to prove these theorems for radial functions in $W^{2,2}(\H^n)$. So, in the rest of this paper, we only work with radial functions in $W^{2,2}(\R^n)$.

\section{Proof of Theorem \ref{keytool}}
In this section, we provide the proof of Theorem \ref{keytool}. We first give some preparations for our proof. Let $u\in W^{2,2}(\H^n)$ be a radial function, and $v:[0,\infty) \to \R$ be a function such that  $u(x) = v(V_g(B_g(0,\rho(x))))$. Let us define the function $u_e$ in $\R^n$ by $u_e(y) = v(\si_n |y|^n)$ with $y \in \R^n$. By a direct computation, we have 
\begin{align}\label{eq:Hnlaplace}
-\De_g u(x) &= -\lt(v''(V_g(B_g(0,\rho(x)))) + \frac{2(n-1)}n \frac{\cosh(\rho(x))}{\sinh(\rho(x))} \frac{v'(V_g(B_g(0,\rho(x))))}{\si_n\sinh(\rho(x))^{n-1}}\rt)\notag\\
&\qquad\qquad \times  (n\si_n)^2 \sinh(\rho(x))^{2(n-1)},
\end{align}
and
\begin{equation}\label{eq:Elaplace}
-\De u_e(y) = -\lt(v''(\si_n |y|^n) + \frac{2(n-1)}n \frac{v'(\si_n |y|^n)}{\si_n |y|^n}\rt) (n \si_n)^2 |y|^{2(n-1)},\quad y\in \R^n.
\end{equation}
Hence, by making the change of variables and using the definition of function $F_n$ (see \eqref{eq:Fnfunction}), we have
\begin{align}\label{eq:L2normDeltaHn}
\int_{\H^n} (\Delta_g u)^2 dV_g &= (n\si_n)^4 \int_0^\infty \lt(v''(s)  + \frac{2(n-1)}{n} \frac{\cosh(F_n(s))}{\sinh(F_n(s))} \frac{v'(s)}{\si_n \sinh(F_n(s))^{n-1}}\rt)^2\notag\\
&\qquad\qquad\qquad\qquad \times  \sinh(F_n(s))^{4(n-1)} ds,
\end{align}
and 
\begin{equation}\label{eq:L2normDeltaE}
\int_{\R^n} (\De u_e)^2 dy = (n\si_n)^4 \int_0^\infty \lt(v''(s) + \frac{2(n-1)}n \frac{v'(s)}{s}\rt)^2 \lt(\frac{s}{\si_n}\rt)^{\frac{4(n-1)}n} ds.
\end{equation} 
In order to prove Theorem \ref{keytool}, we need estimate the quantity
\[
\int_{\H^n} (\Delta_g u)^2 dV_g -\int_{\R^n} (\De u_e)^2 dy.
\]

Let us define
\begin{equation}\label{eq:Phinfunction}
\Phi_n(t) = n \int_0^t \sinh(s)^{n-1} ds, \quad t \geq 0.
\end{equation}
From the definition of $\Phi_n$ and $F_n$, we have $\Phi_n(F_n(t)) = \frac t{\si_n}$. It is easy to see that
\begin{equation}\label{eq:tiemcanPhin}
\Phi_n(t) \sim \sinh(t)^n\quad \text{\rm as}\,\, t \downarrow 0,\quad\text{\rm and}\quad \Phi_n(t) \sim \frac{n}{n-1} \sinh(t)^{n-1}\quad \text{\rm as}\,\, t \to \infty,
\end{equation}
hence it holds
\begin{equation}\label{eq:tiemcanFn}
\sinh(F_n(t))^n \sim
\begin{cases}
\frac{t}{\si_n} &\mbox{as $t \to 0$,}\\
\lt(\frac{n-1} n \frac{t}{\si_n}\rt)^{\frac{n}{n-1}} &\mbox{as $t\to \infty$.}
\end{cases}
\end{equation}
Suppose $-\Delta_g u(x) = f(V_g(B_g(0,\rho(x))))$ for some function $f$ on $[0,\infty)$. Hence, by a simple change of variable, it holds $\int_0^\infty f(s)^2 ds < \infty$. For $t >0$, define
\[
\bar f(t) = \frac1 t\int_0^t f(s) ds.
\]
It is straightforward that
\begin{equation}\label{eq:vformula}
v(t) = \int_t^\infty \frac{s \bar f(s)}{(n \si_n \sinh(F_n(s))^{n-1})^2} ds, \quad t >0.
\end{equation}
We have the following expression of $\int_{\H^n} (\Delta_g u)^2 dV_g -\int_{\R^n} (\De u_e)^2 dy$.
\begin{lemma}\label{expressionbyv'}
It holds
\begin{align}\label{eq:expressionbyv'}
&\frac{\int_{\H^n} (\De_g u)^2 dV_g -\int_{\R^n} (\De u_e)^2 dx}{(n\si_n)^{4} } \notag\\
&= \int_0^\infty\lt(v''(s) + \frac{2(n-1)}n \frac{v'(s)}s\rt)^2 \lt(\sinh(F_n(s))^{4(n-1)} -\lt(\frac{s}{\si_n}\rt)^{\frac{4(n-1)}n} \rt)ds\notag\\
&\quad + \frac{4(n-1)}{n^2} \int_0^\infty v'(s)^2\Bigg(2(n-1) \frac{\cosh(F_n(s))}{s \si_n \sinh(F_n(s))^n}  -\frac{n-1}{2 \si_n^2 \sinh(F_n(s))^{2n-2}}\notag\\
&\qquad\qquad\qquad\qquad\qquad \qquad- \frac{n-2}{2 \si_n^2\sinh(F_n(s))^{2n}} -\frac{3n-2}{2 s^2} \Bigg)\sinh(F_n(s))^{4(n-1)}ds.
\end{align}
\end{lemma}
\begin{proof}
This lemma is proved by using integration by parts. Denote $B(s) = \frac{s \cosh(F_n(s))}{\si_n \sinh(F_n(s))^n}$, we then have
\begin{align}\label{eq:1ststep}
&\frac{\int_{\H^n} |\De_g u|^2 dV_g}{(n\si_n)^4}\notag\\
&\qquad = \int_0^\infty\lt(v''(s) + \frac{2(n-1)}n \frac{v'(s)}s + \frac{2(n-1)}n (B(s) -1) \frac{v'(s)}{s}\rt)^2 \sinh(F_n(s))^{4(n-1)} ds\notag\\
&\qquad =\int_0^\infty\lt(v''(s) + \frac{2(n-1)}n \frac{v'(s)}s\rt)^2 \sinh(F_n(s))^{4(n-1)}ds\notag\\
&\qquad\quad + \frac{4(n-1)}n \int_0^\infty \lt(v''(s) + \frac{2(n-1)}n \frac{v'(s)}s\rt)(B(s) -1) \frac{v'(s)}{s} \sinh(F_n(s))^{4(n-1)} ds\notag\\
&\qquad\quad + \frac{4(n-1)^2}{n^2} \int_0^\infty v'(s)^2 \lt(\frac{B(s)-1}s\rt)^2 \sinh(F_n(s))^{4(n-1)} ds\notag\\
&\qquad =\int_0^\infty\lt(v''(s) + \frac{2(n-1)}n \frac{v'(s)}s\rt)^2 \sinh(F_n(s))^{4(n-1)}ds\notag\\
&\qquad\quad + \frac{4(n-1)^2}{n^2} \int_0^\infty v'(s)^2 \frac{B(s)^2-1}{s^2} \sinh(F_n(s))^{4(n-1)} ds\notag\\
&\qquad\quad + \frac{2(n-1)}n \int_0^\infty ((v'(s))^2)' \frac{(B(s) -1) \sinh(F_n(s))^{4(n-1)}}s ds.
\end{align}
We claim that 
\begin{equation}\label{eq:claimv'0}
\lim_{s\to 0}v'(s)^2\frac{(B(s) -1) \sinh(F_n(s))^{4(n-1)}}s = 0,
\end{equation}
and
\begin{equation}\label{eq:claimv'0a}
\lim_{s\to \infty}v'(s)^2\frac{(B(s) -1) \sinh(F_n(s))^{4(n-1)}}s =0.
\end{equation}
Indeed, we have
\[ 
v'(s) = -\frac{s \bar f(s)}{(n\si_n \sinh(F_n(s))^{n-1})^2)},
\]
hence
\[
v'(s)^2\frac{(B(s) -1) \sinh(F_n(s))^{4(n-1)}}s = \frac{1}{(n\si_n)^4} (B(s) -1) s \bar f(s)^2.
\]
It is easy to check by using \eqref{eq:tiemcanFn} that
\[
\lim_{s\to 0} B(s) =1,\quad \text{\rm and}\quad \lim_{s\to \infty} B(s) = \frac{n}{n-1}.
\]
In other hand, by H\"older inequality we have
\[
|\bar f(s)| \leq \frac{1}{s} \int_0^s |f(t)| dt \leq \frac1{\sqrt{s}} \lt(\int_0^s f(t)^2 dt\rt)^{\frac12},
\]
which implies $\lim_{s\to 0} s \bar f(s)^2 = 0$. The claim \eqref{eq:claimv'0} is then followed. We next consider \eqref{eq:claimv'0a}. For any $\ep >0$, we can take a $t_0 >0$ such that $\int_{t_0}^\infty f(t)^2 dt < \ep^2$. For any $s > t_0$, by H\"older inequality we have
\[
|\bar f(s)| \leq \frac1{t_0}{s} |\bar f(t_0)| + \frac{\sqrt{s-t_0}}{s} \lt(\int_{t_0}^s f(t)^2 dt\rt)^{\frac12} \leq \frac{t_0}{s} |\bar f(t_0)| + \frac{\sqrt{s-t_0}}{s} \ep.
\]
Whence, it holds
\[
\limsup_{s\to \infty} \sqrt{s} |\bar f(s)| \leq \ep,
\]
for any $\ep >0$. Then, we get $\lim_{s\to \infty} s \bar f(s)^2 =0$. This limit and the fact $B(s) \to \frac{n}{n-1}$ as $s \to \infty$ implies the claim \eqref{eq:claimv'0a}.

Applying integration by parts and using the claims \eqref{eq:claimv'0} and \eqref{eq:claimv'0a} to the third in the right hand side of \eqref{eq:1ststep}, we get
\begin{align*}
\int_0^\infty (v'(s)^2)' &\frac{(B(s) -1) \sinh(F_n(s))^{4(n-1)}}s ds\\
&\qquad\qquad = -\int_0^\infty v'(s)^2 \lt(\frac{(B(s) -1) \sinh(F_n(s))^{4(n-1)}}s\rt)' ds.
\end{align*}
Furthermore, from the definition \eqref{eq:Fnfunction} of $F_n$, we have $n \si_n \sinh(F_n(s))^{n-1} F_n'(s) = 1$. Hence, it holds
\begin{align*}
&\lt(\frac{(B(s) -1) \sinh(F_n(s))^{4(n-1)}}s\rt)' \\
&\qquad\qquad\qquad\qquad =\lt(\lt(\frac{B(s)-1}{s}\rt)' + \frac{4(n-1)}n \frac{B(s)^2 -B(s)}{s^2}\rt) \sinh(F_n(s))^{4(n-1)},
\end{align*}
and
\[
\lt(\frac{B(s)}s - \frac1{s}\rt)' = \frac1{n \si_n^2 \sinh(F_n(s))^{2n -2}} - \frac{\cosh(F_n(s))^2}{\si_n^2 \sinh(F_n(s))^{2n}} + \frac1{s^2}.
\]
Plugging the preceding equalities into \eqref{eq:1ststep}, we get
\begin{align*}
&\frac{\int_{\H^n} |\De_g u|^2 dV_g}{(n\si_n)^4} \\
&=\int_0^\infty\lt(v''(s) + \frac{2(n-1)}n \frac{v'(s)}s\rt)^2 \sinh(F_n(s))^{4(n-1)}ds\\
&\quad - \frac{4(n-1)^2}{n^2} \int_0^\infty v'(s)^2\Bigg(\frac{n}{2(n-1)}\lt(\frac{B(s)-1}{s}\rt)' +\frac{B(s)^2 -2B(s) +1}{s^2} \Bigg)\\
&\qquad\qquad\qquad\qquad\qquad\qquad \times \sinh(F_n(s))^{4(n-1)} ds\\
&=\int_0^\infty\lt(v''(s) + \frac{2(n-1)}n \frac{v'(s)}s\rt)^2 \sinh(F_n(s))^{4(n-1)}ds\\
&\quad + \frac{4(n-1)}{n^2} \int_0^\infty v'(s)^2\Bigg(2(n-1) \frac{\cosh(F_n(s))}{s \si_n \sinh(F_n(s))^n}  -\frac{n-1}{2 \si_n^2 \sinh(F_n(s))^{2n-2}}\\
&\qquad\qquad\qquad\qquad\qquad \qquad\qquad\qquad - \frac{n-2}{2 \si_n^2\sinh(F_n(s))^{2n}} -\frac{3n-2}{2 s^2} \Bigg)\sinh(F_n(s))^{4(n-1)}ds.
\end{align*}
The desired equality \eqref{eq:expressionbyv'} is now followed from \eqref{eq:L2normDeltaE} and the previous equality.
\end{proof}
To proceed, we shall need the following estimate which is the key in the proof of Theorem \ref{keytool}.
\begin{lemma}\label{keyestimate}
Let $n \geq 4$. There holds
\begin{align}\label{eq:keyestimate}
&\Bigg(2(n-1) \frac{\cosh(F_n(s))}{s \si_n \sinh(F_n(s))^n}  -\frac{n-1}{2 \si_n^2 \sinh(F_n(s))^{2n-2}} \notag\\
&\qquad\qquad - \frac{n-2}{2 \si_n^2\sinh(F_n(s))^{2n}} -\frac{3n-2}{2 s^2} \Bigg)\sinh(F_n(s))^{4(n-1)} > \frac{(n-1)^3(n-2)}{2 n^4} \frac{s^2}{\si_n^4}, 
\end{align}
for any $s >0$.
\end{lemma}
\begin{proof}
We start the proof by remark that the inequality \eqref{eq:keyestimate} is equivalent to
\begin{align}\label{eq:arbitraryn}
\frac{(n-1)^3(n-2)}{2 n^4}& \Phi_n(t)^4+\lt(\frac{n-1}2 \sinh(t)^{2(n-1)} + \frac{n-2}2\sinh(t)^{2n-4}\rt)\Phi_n(t)^2 \notag\\
&\quad -2(n-1) \sinh(t)^{3n-4} \cosh(t) \Phi_n(t) + \frac{3n-2}2 \sinh(t)^{4(n-1)} < 0,
\end{align}
for any $t>0$ where $\Phi_n$ is defined by \eqref{eq:Phinfunction}. Let us define
\begin{align*}
G_n(t) &= \frac{(n-1)^3(n-2)}{2 n^4} \Phi_n(t)^4+\lt(\frac{n-1}2 \sinh(t)^{2(n-1)} + \frac{n-2}2\sinh(t)^{2n-4}\rt)\Phi_n(t)^2 \notag\\
&\quad -2(n-1) \sinh(t)^{3n-4} \cosh(t) \Phi_n(t) + \frac{3n-2}2 \sinh(t)^{4(n-1)}.
\end{align*}
We need to check $G_n(t) \leq 0$ for $t \geq 0$. Notice that $G_n(0) = 0$. Differentiating the function $G_n$ gives
\begin{align*}
G_n'(t)&= \frac{2(n-1)^3(n-2)}{n^3} \Phi_n(t)^3 \sinh(t)^{n-1}\\
&\quad + \lt((n-1)^2 \sinh(t)^{2n-3} \cosh(t)+ (n-2)^2\sinh(t)^{2n-5} \cosh(t)\rt)\Phi_n(t)^2\\
&\quad + n\lt((n-1) \sinh(t)^{2(n-1)} + (n-2)\sinh(t)^{2n-4}\rt)\Phi_n(t) \sinh(t)^{n-1}\\
&\quad -2(n-1) ((3n -3) \sinh(t)^{3n-3} + (3n -4)\sinh(t)^{3n-5})\Phi_n(t)\\
&\quad -2n(n-1) \sinh(t)^{4n-5} \cosh(t) + 2(n-1)(3n -2) \sinh(t)^{4n-5} \cosh(t)\\
&= \frac{2(n-1)^3(n-2)}{n^3} \Phi_n(t)^3 \sinh(t)^{n-1}\\
&\quad + \lt((n-1)^2 \sinh(t)^{2n-3} \cosh(t)+ (n-2)^2\sinh(t)^{2n-5} \cosh(t)\rt)\Phi_n(t)^2\\
&\quad -((n-1)(5n -6) \sinh(t)^{3n -3} + (5n^2 -12n +8) \sinh(t)^{3n-5}) \Phi_n(t)\\
&\quad + 4(n-1)^2 \sinh(t)^{4n-5} \cosh(t).
\end{align*}
Denote
\begin{align*}
H_n(t)&= \frac{2(n-1)^3(n-2)}{n^3} \Phi_n(t)^3\\
&\quad + \lt((n-1)^2 \sinh(t)^{n-2} \cosh(t)+ (n-2)^2\sinh(t)^{n-4} \cosh(t)\rt)\Phi_n(t)^2\\
&\quad -((n-1)(5n -6) \sinh(t)^{2n -2} + (5n^2 -12n +8) \sinh(t)^{2n-4}) \Phi_n(t)\\
&\quad + 4(n-1)^2 \sinh(t)^{3n-4} \cosh(t),
\end{align*}
we then have $G_n'(t) = H_n(t) \sinh(t)^{n-1}$ and $H_n(0) = 0$. Differentiating $H_n$ once, we get
\begin{align*}
&H_n'(t) \\
& = \frac{6(n-1)^3(n-2)}{n^2} \Phi_n(t)^2 \sinh(t)^{n-1}\\
& + \lt((n-1)^3 + \frac{(n-2)(2n^2 -7n +7)}{\sinh(t)^2} + \frac{(n-2)^2(n-4)}{\sinh(t)^4}\rt) \Phi_n(t)^2 \sinh(t)^{n-1}\\
& + 2n \lt((n-1)^2 \sinh(t)^{n-2} \cosh(t)+ (n-2)^2\sinh(t)^{n-4} \cosh(t)\rt)\Phi_n(t) \sinh(t)^{n-1}\\
& -(2(n-1)^2(5n -6) \sinh(t)^{2n -3} + (2n-4)(5n^2 -12n +8) \sinh(t)^{2n-5}) \cosh(t) \Phi_n(t)\\
& - n((n-1)(5n -6) \sinh(t)^{2n -2} + (5n^2 -12n +8) \sinh(t)^{2n-4}) \sinh(t)^{n-1}\\
& + 4(n-1)^2 (3(n-1)\sinh(t)^{3n-3} + (3n -4) \sinh(t)^{3n -5})\\
&= \lt(\frac{(n-1)^3 (n^2 +6n -12)}{n^2}+ \frac{(n-2)(2n^2 -7n +7)}{\sinh(t)^2} + \frac{(n-2)^2(n-4)}{\sinh(t)^4}\rt)  \Phi_n(t)^2 \sinh(t)^{n-1}\\
& -(4(n-1)^2(2n-3) \sinh(t)^{2} + 4(n-2)(2n^2 -5n +4)) \cosh(t) \Phi_n(t) \sinh(t)^{2n-5}\\
& + ((n-1)(7n^2 -18 n+12) \sinh(t)^{2} + (7n^3 -28n^2 + 36n -16)) \sinh(t)^{3n-5}.
\end{align*}
Denote
\begin{align*}
J_n(t)&= \Phi_n(t)^2 + \frac{((n-1)(7n^2 -18n + 12)\sinh(t)^2 + 7n^3 -28 n^2 + 36 n -16)\sinh(t)^{2n}}{\frac{(n-1)^3(n^2 + 6n -12)}{n^2} \sinh(t)^4 + (n-2)(2n^2 -7n + 7)\sinh(t)^2 + (n-2)^2(n-4)}\\
& -\frac{(4(n-1)^2(2n-3) \sinh(t)^{2} + 4(n-2)(2n^2-5n +4))\sinh(t)^{n}\cosh(t)}{\frac{(n-1)^3(n^2 + 6n -12)}{n^2} \sinh(t)^4 + (n-2)(2n^2 -7n + 7)\sinh(t)^2 + (n-2)^2(n-4)} \Phi_n(t).
\end{align*}
We have 
\begin{align*}
H_n'(t) &= \Bigg(\frac{(n-1)^3 (n^2 +6n -12)}{n^2}+ \frac{(n-2)(2n^2 -7n +7)}{\sinh(t)^2} + \frac{(n-2)^2(n-4)}{\sinh(t)^4}\Bigg)\\
&\qquad\qquad \times J_n(t)\, \sinh(t)^{n-1},
\end{align*}
and $J_n(0) =0$. Denote $s =\sinh(t)^2$, differentiating $J_n$ we get
\begin{align*}
&\frac{J_n'(t)}{\sinh(t)^{n-1}}\notag\\
& = 2n\Phi_n(t) -n\frac{(4(n-1)^2(2n-3) s + 4(n-2)(2n^2-5n +4))\sinh(t)^{n}\cosh(t)}{\frac{(n-1)^3(n^2 + 6n -12)}{n^2} s^2 + (n-2)(2n^2 -7n + 7)s + (n-2)^2(n-4)} \notag\\
& - \frac{4(n-1)^2(2n -3)(n+3)s^2 + 4(4n^4 -10 n^3-n^2 + 19n -14)s + 4(n-2)n(2n^2 -5n +4)}{\frac{(n-1)^3(n^2 + 6n -12)}{n^2} s^2 + (n-2)(2n^2 -7n + 7)s + (n-2)^2(n-4)}\Phi_n(t)\notag\\
& +\frac{(4(n-1)^2(2n-3)s + 4(n-2)(2n^2 -5n +4))(4 \frac{(n-1)^3(n^2 + 6n -12)}{n^2} s + 2(n-2)(2n^2 -7n +7))}{(\frac{(n-1)^3(n^2 + 6n -12)}{n^2} s^2 + (n-2)(2n^2 -7n + 7)s + (n-2)^2(n-4))^2}\notag\\
&\qquad \times s(s+1) \Phi_n(t)\notag\\
& + \frac{(2(n^2 -1)(7n^2 -18n + 12)s + 2n(7n^3 -28n^2 + 36n -16))\sinh(t)^n\cosh(t)}{\frac{(n-1)^3(n^2 + 6n -12)}{n^2} s^2 + (n-2)(2n^2 -7n + 7)s + (n-2)^2(n-4)}\notag\\
& -\frac{((n -1)(7n^2 -18n + 12)s + 7n^3 -28n^2 + 36n -16)(4 \frac{(n-1)^3(n^2 + 6n -12)}{n^2} s + 2(n-2)(2n^2 -7n +7))}{(\frac{(n-1)^3(n^2 + 6n -12)}{n^2} s^2 + (n-2)(2n^2 -7n + 7)s + (n-2)^2(n-4))^2}\notag\\
&\qquad \times s \sinh(t)^n \cosh(t).
\end{align*}
By direct computations, we simplify $\frac{J_n'(t)}{\sinh(t)^{n-1}}$ as
\begin{align}\label{eq:daohamJn}
&\frac{J_n'(t)}{\sinh(t)^{n-1}}\notag\\
&=-\frac{\frac{2(n-1)^2(3n^3+ n^2 -12)}{n}s ^2 + (12 n^4 -18n^3 -46 n^2+ 104 n -56)s + 2n^2(n-2)(3n-4)}{\frac{(n-1)^3(n^2 + 6n -12)}{n^2} s^2 + (n-2)(2n^2 -7n + 7)s + (n-2)^2(n-4)}\Phi_n(t)\notag\\
&+\frac{(4(n-1)^2(2n-3)s + 4(n-2)(2n^2 -5n +4))(4 \frac{(n-1)^3(n^2 + 6n -12)}{n^2} s + 2(n-2)(2n^2 -7n +7))}{(\frac{(n-1)^3(n^2 + 6n -12)}{n^2} s^2 + (n-2)(2n^2 -7n + 7)s + (n-2)^2(n-4))^2}\notag\\
&\qquad \times s(s+1) \Phi_n(t)\notag\\
& + \frac{2(n-1)(3n^3 -n^2 -12n + 12)s + 2n^2(n-2)(3n-4)}{\frac{(n-1)^3(n^2 + 6n -12)}{n^2} s^2 + (n-2)(2n^2 -7n + 7)s + (n-2)^2(n-4)}\sinh(t)^n \cosh(t)\notag\\
&-\frac{((n -1)(7n^2 -18n + 12)s + 7n^3 -28n^2 + 36n -16)(4 \frac{(n-1)^3(n^2 + 6n -12)}{n^2} s + 2(n-2)(2n^2 -7n +7))}{(\frac{(n-1)^3(n^2 + 6n -12)}{n^2} s^2 + (n-2)(2n^2 -7n + 7)s + (n-2)^2(n-4))^2}\notag\\
&\qquad \times s \sinh(t)^n \cosh(t)\notag\\
&=-\frac{A(s) \Phi_n(t) - B(s)\sinh(t)^n \cosh(t)}{(\frac{(n-1)^3(n^2 + 6n -12)}{n^2} s^2 + (n-2)(2n^2 -7n + 7)s + (n-2)^2(n-4))^2},
\end{align}
where
\begin{align*}
A(s)& =\frac{6(n-1)^6(n-2)^2(n^2 + 6n-12)}{n^3}s^4 \\
&\quad + \frac{(n-1)^2}{n^2}(24 n^7 -116 n^6-88 n^5 + 1732 n^4 - 4920 n^3 +6744 n^2 -4800 n +1440) s^3\\
&\quad + \frac{n-2}{n^2} (36 n^8 -252 n^7 + 506 n^6 + 392 n^5-3386 n^4 + 6504 n^3 -6480 n^2 + 3456n -768) s^2 \\
&\quad + (n-2)^2(24 n^5 -156 n^4 + 316 n^3 -224 n^2 + 32 n) s\\
&\quad + 2n^2(n-2)^3(3n-4)(n-4),
\end{align*}
and
\begin{align*}
B(s) &= \frac{6(n-1)^5(n^2 + 6n -12)(n-2)^2}{n^2} s^3\\
&\quad + \frac{2(n-1)(n-2)}{n^2} (9 n^7 -43 n^6 -24 n^5 + 502 n^4 -1242 n^3 +1424 n^2 - 816 n +192) s^2\\
&\quad + 2(n-2)^2(9 n^5 -59 n^4 + 131 n^3 -123 n^2 + 46n -8) s\\
&\quad + 2 n^2(n-2)^3(n-4)(3n -4).
\end{align*}

We will need the following lower bound for $\Phi_n(t)$
\begin{equation}\label{eq:lowerboundPhin}
\int_0^t \sinh(s)^{n-1} ds \geq \frac{\sinh(t)^n \cosh(t)((n-3)\sinh(t)^2 + n +2)}{(n-1)(n-3) \sinh(t)^4 + 2n(n-1)\sinh(t)^2 + n(n+2)},
\end{equation}
for any $n \geq 4$. Indeed, let us define the function $K_n$ by
\[
K_n(t) = \int_0^t \sinh(s)^{n-1} ds - \frac{\sinh(t)^n \cosh(t)((n-3)\sinh(t)^2 + n +2)}{(n-1)(n-3) \sinh(t)^4 + 2n(n-1)\sinh(t)^2 + n(n+2)}.
\]
We have $K_n(0) =0$. Differentiating $K_n$ gives
\begin{align*}
K_n'(t) &= \sinh(t)^{n-1} -\frac{(n\sinh(t)^{n-1}\cosh(t)^2 + \sinh(t)^{n+1})((n-3)\sinh(t)^2 + (n+2))}{(n-1)(n-3) \sinh(t)^4 + 2n(n-1)\sinh(t)^2 + n(n+2)}\\
&\quad -\frac{2(n-3)\sinh(t)^{n+1} \cosh(t)^2}{(n-1)(n-3) \sinh(t)^4 + 2n(n-1)\sinh(t)^2 + n(n+2)}\\
&\quad + \frac{4(n-1)\sinh(t)^{n+1} \cosh(t)^2 ((n-3)\sinh(t)^2 + n+2)((n-3)\sinh(t)^2 + n)}{((n-1)(n-3) \sinh(t)^4 + 2n(n-1)\sinh(t)^2 + n(n+2))^2}.
\end{align*}
Denote again $s = \sinh^2(t)$, we have
\begin{align*}
\frac{K_n'(t)}{\sinh(t)^{n-1}}& = 1 -\frac{((n+1)s + n)((n-3)s + n+2) + 2(n-3)s(s+1)}{(n-1)(n-3)s^2 + 2n(n-1)s + n(n+2)}\\
&\quad + 4(n-1) \frac{s(s+1)((n-3)s + n+2)((n-3)s +n)}{((n-1)(n-3)s^2 + 2n(n-1)s + n(n+2))^2}\\
&= \frac{24 s^2}{((n-1)(n-3)s^2 + 2n(n-1)s + n(n+2))^2}\\
&\geq 0.
\end{align*}
Hence $K_n(t) \geq K_n(0) = 0$ for $t\geq 0$ as desired.

We next claim that
\begin{equation}\label{eq:claimkey}
n((n-3)s + n+2) A(s) > ((n-1)(n-3)s^2 + 2n(n-1)s + n(n+2))B(s),\, s >0.
\end{equation}
The claim \eqref{eq:claimkey}, \eqref{eq:lowerboundPhin} and \eqref{eq:daohamJn} implies $J_n'(t) < 0$ for any $t >0$. Therefore, $J_n$ is strictly decreasing in $[0,\infty)$ and $J_n(t) < J_n(0) = 0$ for any $t > 0$.  The relation between $J_n$ and $H_n'$ yields $H_n'(t) < 0$ for any $t >0$ which gives $H_n(t) < H_n(0) =0$ for any $t >0$. Whence $G_n'(t) = \sinh^{n-1}(t) \, H_n(t) < 0$ for any $t > 0$ which forces $G_n(t) < G_n(0) =0$ for any $t >0$ as desired.

It remains to verify the claim \eqref{eq:claimkey}. The proof is the long and complicated computations. Define
\[
P(s) = n((n-3)s + n+2) A(s)- ((n-1)(n-3)s^2 + 2n(n-1)s + n(n+2))B(s).
\]
$P$ is a polynomial of $s$ of order $5$, hence we can write 
\[
P(s) = a_0(n) + a_1(n) s + a_2(n) s^2 + a_3(n) s^3 + a_4(n) s^4 + a_5(n) s^5,
\]
with coefficients $a_i(n), i= 0, \ldots,5$ depending on $n$. Our way is to show $a_i(n) \geq 0$ for $n\geq 4$. Indeed, $a_0(n) = a_5(n) = 0$. The direct computations give
\begin{align*}
a_4(n)& = \frac{2 (n - 1)^2 (n - 2)}{n^2}(2 n - 3) (3 n - 5) (n^2 + 6 n - 12) (n^3 - 5 n^2 + 6 n - 4)\\
&\quad >0,
\end{align*}
\begin{align*}
a_3(n)& =\frac{2(n-2)}n (18 n^8 - 135 n^7 + 485 n^6 - 1087 n^5 + 1677 n^4 - 1926 n^3 + 1656 n^2 - 888 n + 192)\\
&\quad > 0,
\end{align*}
\begin{align*}
a_2(n) &= \frac{2(n-2)^2}n (18 n^7 - 123 n^6 + 370 n^5 -551 n^4 + 258 n^3 + 360 n^2 -528 n + 192) > 0,
\end{align*}
and
\begin{align*}
a_1(n) & =2 n(n - 2)^3 (n-4)(6 n^3 -n^2 -5 n + 2) \geq 0.
\end{align*}
Hence the claim \eqref{eq:claimkey} is proved, and the proof of this lemma is completely finished.
\end{proof}

We are now ready to prove Theorem \ref{keytool}.
\begin{proof}[Proof of Theorem \ref{keytool}]
It follows from Lemmas \ref{expressionbyv'} and \ref{keyestimate} that
\begin{align}\label{eq:proofTh2.2}
\int_{\H^n} &(\De_g u)^2 dV_g -\int_{\R^n} (\De u_e)^2 dx\notag\\
&\geq (n\si_n)^{4}\int_0^\infty\lt(v''(s) + \frac{2(n-1)}n \frac{v'(s)}s\rt)^2 \lt(\sinh(F_n(s))^{4(n-1)} -\lt(\frac{s}{\si_n}\rt)^{\frac{4(n-1)}n} \rt)ds\notag\\
&\quad + \frac{2(n-1)^4(n-2)}{ n^4}  \int_0^\infty v'(s)^2 (n\si_n)^2 \lt(\frac{s}{\si_n}\rt)^2ds.
\end{align}
It was proved in \cite{VHN2018,VHN2018a} that 
\[
\sinh(F_n(s))^{4(n-1)} -\lt(\frac{s}{\si_n}\rt)^{\frac{4(n-1)}n}  \geq \lt(\frac{n-1}n\rt)^4 \lt(\frac{s}{\si_n}\rt)^4.
\]
Plugging the preceding estimate into \eqref{eq:proofTh2.2} gives
\begin{align}\label{eq:tofinish}
\int_{\H^n} (\De_g u)^2 dV_g & -\int_{\R^n} (\De u_e)^2 dx\notag\\
&\geq \frac{(n-1)^4}{n^4}\int_0^\infty\lt(v''(s) + \frac{2(n-1)}n \frac{v'(s)}s\rt)^2 (n\si_n)^4 \lt(\frac s{\si_n}\rt)^4ds\notag\\
&\quad + \frac{2(n-1)^4(n-2)}{ n^4}  \int_0^\infty v'(s)^2 (n\si_n)^2 \lt(\frac{s}{\si_n}\rt)^2ds\notag\\
&=\frac{(n-1)^4}{n^4}\lt( \int_{\R^n} (\De u_e)^2 |x|^4 dx + 2(n-2) \int_{\R^n} |\na u_e|^2 |x|^2 dx\rt). 
\end{align}
Recall the weighted Rellich and Hardy inequality in $\R^n$ (see, e. g, \cite[Theorems $12$ and $13$]{DH})
\[
\int_{\R^n} (\De u_e)^2 |x|^4 dx \geq \frac{n^2(n-4)^2}{16} \int_{\R^n} |u_e|^2 dx,
\]
and
\[
\int_{\R^n} |\na u_e|^2 |x|^2 dx \geq \frac{n^2}{4} \int_{\R^n} |u_e|^2 dx.
\]
Applying these inequalities to the right hand side of \eqref{eq:tofinish}, we get
\begin{align*}
\int_{\H^n} (\De_g u)^2 dV_g -\int_{\R^n} (\De u_e)^2 dx &\geq \frac{(n-1)^4}{n^4} \lt(\frac{n^2(n-4)^2}{16} + 2(n-2)\frac{n^2}{4}\rt)  \int_{\R^n} |u_e|^2 dx\\
&= \frac{(n-1)^4}{16} \int_{\R^n} |u_e|^2 dx\\
&= \frac{(n-1)^4}{16} \int_{\H^n} u^2 dV_g,
\end{align*}
as desired, here we use \eqref{eq:dangthucnorm} for the last equality.
\end{proof}

\section{Proofs of the main theorems}
In this section, we provide the proof of Theorems \ref{Rellich}, \ref{Sobolev} and \ref{Adams}. We begin by proving Theorem \ref{Rellich}.

\subsection{Proof of Theorem \ref{Rellich}}
This subsection is devote to prove Theorem \ref{Rellich}. The proof follows from Theorem \ref{keytool} and the sharp Rellich inequality in $\R^n$ as follows.
\begin{proof}[Proof of Theorem \ref{Rellich}]
As discussed in Section $2$, it is enough to prove Theorem \ref{Rellich} for radial function $u \in W^{2,2}(\H^n)$. Suppose $u \in W^{2,2}(\H^n)$ to be a radial function, and $v$ is a function on $[0,\infty)$ such that $u(x) = v(V_g(B_g(0,\rho(x))))$. Define the function $u_e$ in $\R^n$ by $u_e(y)= v(\si_n |y|^n)$. By Theorem \ref{keytool}, we have
\[
\int_{\H^n} (\De_g u)^2 dV_g - \frac{(n-1)^4}{16} \int_{\H^n} u^2 dV_g \geq \int_{\R^n} (\De u_e)^2 dy.
\]
Since $n\geq 5$, by Rellich inequality (see \cite{DH}), we have
\begin{align*}
\int_{\R^n} (\De u_e)^2 dy &\geq \frac{n^2(n-4)^2}{16} \int_{\R^n} \frac{u_e(y)^2}{|y|^4} dy\\
& = \frac{n^2(n-4)^2}{16} \int_0^\infty \frac{v(s)^2}{(\frac s{\si_n})^{\frac4n}} ds\\
&= \frac{n^2(n-4)^2}{16} \int_{\H^n} \frac{u(x)^2}{(\frac{V_g(B_g(0,\rho(x)))}{\si_n})^{\frac4n}} dV_g.
\end{align*}
Combining these inequalities finishes the proof of Theorem \ref{Rellich}.
\end{proof}
\subsection{Proof of Theorem \ref{Sobolev}}
The proof of Theorem \ref{Sobolev} is similar with the one of Theorem \ref{Rellich}. Instead of using Rellich inequality, we use the sharp Sobolev inequality \eqref{eq:Sobolevorder2} as follows.
\begin{proof}[Proof of Theorem \ref{Sobolev}]
As discussed in Section $2$, it is enough to prove Theorem \ref{Sobolev} for radial function $u \in W^{2,2}(\H^n)$. Suppose $u \in W^{2,2}(\H^n)$ to be a radial function, and $v$ is a function on $[0,\infty)$ such that $u(x) = v(V_g(B_g(0,\rho(x))))$. Define the function $u_e$ in $\R^n$ by $u_e(y)= v(\si_n |y|^n)$. By Theorem \ref{keytool}, we have
\[
\int_{\H^n} (\De_g u)^2 dV_g - \frac{(n-1)^4}{16} \int_{\H^n} u^2 dV_g \geq \int_{\R^n} (\De u_e)^2 dy.
\]
Since $n\geq 5$, by the sharp Sobolev inequality \eqref{eq:Sobolevorder2}, we have
\begin{align*}
\int_{\R^n} (\De u_e)^2 dy &\geq S_2(n,2) \lt(\int_{\R^n} |u_e(y)|^{\frac{2n}{n-4}} dy\rt)^{\frac{n-4}n}\\
& = S_2(n,2) \lt(\int_0^\infty |v(s)|^{\frac{2n}{n-4}} ds\rt)^{\frac{n-4}n}\\
&= S_2(n,2) \lt(\int_{\H^n} |u|^{\frac{2n}{n-4}} dV_g\rt)^{\frac{n-4}n}.
\end{align*}
Combining these inequalities finishes the proof of Theorem \ref{Sobolev}.
\end{proof}

As mentioned in introduction, we obtain the following Poincar\'e--Sobolev inequality for the GJMS operator $P_2$ from Theorem \ref{Sobolev} with an explicit constant.
\begin{corollary}\label{cor}
Let $n \geq 5$. It holds
\begin{equation}\label{eq:PS2th}
\int_{\H^n} (P_2 u) u\, dV_g - \frac{9}{16} \int_{\H^n} u^2 dV_g \geq \frac{5S_2(n,2)}{(n-1)^2} \lt(\int_{\H^n} |u|^{\frac{2n}{n-4}} dV_g\rt)^{\frac{n-4}n}, \quad u\in C_0^\infty(\H^n).
\end{equation}
\end{corollary}
The inequality \eqref{eq:PS2th} is exact the inequality \eqref{eq:LuYangineq} for $k =2$ which improves the Poincar\'e inequality \eqref{eq:PoincarePk} for $k=2$. Our proof of \eqref{eq:PS2th} in this paper is completely different with the one of Lu and Yang \cite{LuYanga}. Moreover, we obtain the explicit constant in the right hand side of \eqref{eq:PS2th}.
\begin{proof}[Proof of Corollary \ref{cor}]
For $u \in C_0^\infty(\H^n)$, we have 
\begin{align*}
\int_{\H^n} (P_2 u) u\, dV_g - \frac{9}{16} \int_{\H^n} u^2 dV_g&= \int_{\H^n} (\Delta_g u)^2 dV_g + \lt(\frac{n(n-2)}2 -2\rt) \int_{\H^n} (\Delta_g u) u\, dV_g \\
&\quad + \frac{n(n-2) (n(n-2) -8) -9}{16} \int_{\H^n} u^2 dV_g.
\end{align*}
Hence, it holds
\begin{align*}
\int_{\H^n} (P_2 u) u\, dV_g &- \frac{9}{16} \int_{\H^n} u^2 dV_g - \frac{5}{(n-1)^2} \lt(\int_{\H^n} (\De_g u)^2 dV_g - \frac{(n-1)^4}{16} \int_{\H^n} u^2 dV_g\rt)\\
&= \frac{n(n-2) -4}{(n-1)^2} \int_{\H^n} (\Delta_g u)^2 dV_g + \frac{n(n-2)-4}2 \int_{\H^n} (\Delta_g u) u\, dV_g \\
&\quad + \frac{(n(n-2) -4)(n-1)^2}{16} \int_{\H^n} u^2 dV_g\\
&= \frac{n(n-2) -4}{(n-1)^2} \int_{\H^n} \lt(\Delta_g u + \frac{(n-1)^2}{4} u\rt)^2 dV_g\\
&\geq 0.
\end{align*}
Consequence, we get
\[
\int_{\H^n} (P_2 u) u\, dV_g - \frac{9}{16} \int_{\H^n} u^2 dV_g \geq \frac{5}{(n-1)^2} \lt(\int_{\H^n} (\De_g u)^2 dV_g - \frac{(n-1)^4}{16} \int_{\H^n} u^2 dV_g\rt).
\]
This corollary follows from the preceding inequality and Theorem \ref{Sobolev}.
\end{proof}

It is well known that the sharp Sobolev inequality and Rellich inequality are special cases of family of Rellich--Sobolev inequalities in $\R^n$ (see, e. g., \cite{Musina,MusinaS}) as follows
\[
\int_{\R^n} (\De w)^2 dy \geq S_{2,s}(n,2) \lt(\int_{\R^n} \frac{|w(y)|^{\frac{2(n-s)}{n-4}}}{|y|^s} dy\rt)^{\frac{n-4}{n-s}}, \quad w \in C_0^\infty(\R^n), \, 0 \leq s \leq 4,
\]
where $S_{2,s}(n,2)$ denotes the sharp constant in the inequality above. Obviously, we have $S_{2,0}(n,2) = S_2(n,2)$ and $S_{2,4}(n,2) = \frac{n^2(n-4)^2}{16}$. Using the Rellich--Sobolev inequality above and the argument in the proof of Theorems \ref{Rellich} and \ref{Sobolev}, we easily obtain the following inequality which combines the sharp Poincar\'e inequality and the sharp Rellich--Sobolev inequality in $\H^n$,
\[
\int_{\H^n} (\De_g u)^2 dV_g - \frac{(n-1)^4}{16} \int_{\H^n} u^2 dV_g \geq S_{2,s}(n,2) \lt(\int_{\H^n} \frac{|u(x)|^{\frac{2(n-s)}{n-4)}}}{(\frac{V_g(B(0,\rho(x)))}{\si_n})^{\frac{s}n}} dV_g\rt)^{\frac{n-4}{n-s}},
\]
for any $u\in C_0^\infty(\H^n)$, $0 \leq s \leq 4$.
\subsection{Proof of Theorem \ref{Adams}}
In this section, we provide the proof of Theorem \ref{Adams}. Similar to the proof of Theorems \ref{Rellich} and \ref{Sobolev}, the proof of Theorem \ref{Adams} is done by using Theorem \ref{keytool} and the sharp Adams inequality \eqref{eq:AdamsR4} in $\R^4$. By scaling, it is easy to derive from \eqref{eq:AdamsR4} the following inequality
\begin{equation}\label{eq:Adamstau}
\sup_{\int_{\R^4} (\De u)^2 + \tau \int_{\R^4} u^2 dx \leq 1} \int_{\R^4} \lt(e^{32\pi^2 u^2} -1\rt) dx \leq \frac1{\tau}\, \sup_{\int_{\R^4} (\De u)^2 + \int_{\R^4} u^2 dx \leq 1} \int_{\R^4} \lt(e^{32\pi^2 u^2} -1\rt) dx.
\end{equation}
\begin{proof}[Proof of Theorem \ref{Adams}]
Again, as discussed in Section $2$, we only need to prove \eqref{eq:Adams} for radial function $u \in W^{2,2}(\H^4)$. Let $\lam < \frac{81}{16}$ and $u \in W^{2,2}(\H^4)$ be a radial function such that
\[
\int_{\H^n} (\Delta_g u)^2 dV_g - \lam \int_{\H^n} u^2 dV_g \leq 1.
\]
Suppose $u(x) = v(V_g(B_g(0,\rho(x))))$ for a function $v$ on $[0,\infty)$, let us define the function $u_e$ in $\R^4$ by $u_e(y) = v(\si_4 |y|^4)$ for $y\in \R^4$. By Theorem \ref{keytool}, we have
\[
\int_{\H^4} (\De_g u)^2 dV_g - \frac{81}{16} \int_{\H^4} u^2 dV_g \geq \int_{\R^4} (\De u_e)^2 dy,
\]
which then implies
\[
\int_{\R^4} (\De u_e)^2 dy + \lt(\frac{81}{16} -\lam\rt) \int_{\R^4} u_e^2 dy \leq \int_{\H^n} (\Delta_g u)^2 dV_g - \lam \int_{\H^n} u^2 dV_g \leq 1.
\]
Since $\lam < \frac{81}{16}$, then $\tau: =\frac{81}{16} -\lam >0$. By the sharp Adams inequality \eqref{eq:Adamstau} for this $\tau$, we have
\[
\int_{\R^n} \lt(e^{32 \pi^2 u_e^2} -1\rt) dy \leq \frac{16}{81 -16 \lam}\,\sup_{\int_{\R^4} (\De u)^2 + \int_{\R^4} u^2 dx \leq 1} \int_{\R^4} \lt(e^{32\pi^2 u^2} -1\rt) dx=: C_\lam.
\]
However, we have
\[
\int_{\R^n} \lt(e^{32 \pi^2 u_e^2} -1\rt) dy = \int_0^\infty \lt(e^{32 \pi^2 v(s)^2} -1\rt)^2 = \int_{\H^n} \lt(e^{32\pi u^2} -1\rt) dV_g.
\]
The proof of \eqref{eq:Adams} is then completed.

The proof of \eqref{eq:Adamsexact} is completely similar with the one of \eqref{eq:Adams} by using Theorem \ref{keytool} and the Adams inequality with exact growth in $\R^4$ due to Masmoudi and Sani. So we omit it. It remains to check the optimality of the power $2$ in the denominator in the integral in \eqref{eq:Adamsexact}. To do this, we construct a sequence of test functions as follows. For each $m \in \N$, consider
\[
u_m(x) = 
\begin{cases}
\sqrt{\frac1{32 \pi^2} \ln m} + \sqrt{\frac{1}{8 \pi^2 \ln m}}(1 - \sqrt{m} |x|^2)&\mbox{if $|x|\leq m^{-\frac14}$,}\\
-\sqrt{\frac1{2\pi^2 \ln m}} \ln |x| &\mbox{if $m^{-\frac14} \leq |x| \leq 1$,}\\
\xi_m(x) &\mbox{if $|x| \geq 1$,}
\end{cases}
\]
where $\xi_m \in C_0^\infty(\R^n)$ is a radial function such that $\xi_m(x) =0$ for $|x|\geq 2$, 
\[
\xi_m\Big{|}_{\{|x| =1\}} =0,\qquad \frac{\pa \xi_m}{\pa r}\Big{|}_{\{|x| =1\}} = -\sqrt{\frac{1}{2\pi^2 \ln m}},
\]
and $\xi_m, |\na \xi_m|$ and $\De \xi_m$ are all $O((\ln m)^{-\frac12})$. The choice of the sequence $\{\xi_m\}_m$ is inspired by the sequence of Masmoudi and Sani \cite{MS2014}. Following the idea of Karmakar \cite{Karmakar}, we set $\bar u_m(x) = u_m(3 x)$ then the support of $\bar u_m$ is contained in $\{|x|\leq \frac23\}$. We can easily check that
\[
\int_{\H^4} \bar u^2 dV_g = O((\ln m)^{-1}),\quad \text{and}\quad \int_{\H^4} (\De_g \bar u)^2 dV_g = 1 + O((\ln m)^{-1}).
\]
Let $w_m = c_m \bar u_m$ where $c_m$ is chosen such that $\int_{\H^4} (\De_g w_m)^2 dV_g  - \frac{81}{16} \int_{\H^4} w_m^2 dV_g =1$. Then $c_m = 1 + O((\ln m)^{-1})$. Notice that $w_m$ is radial function. Moreover, we have
\begin{align*}
\int_{\H^4} \frac{e^{32\pi^2 w_m^2} -1}{(1 + w_m^2)^p} dV_g &\geq  \int_{\{|x|\leq 3^{-1} m^{-\frac14}\}}\frac{e^{32\pi^2 w_m} -1}{(1 + w_m)^p} dV_g\\
&\geq C (\ln m)^{-\frac p2} (1 + O((\ln m)^{-1})  \int_{\{|x|\leq 3^{-1} m^{-\frac14}\}} (e^{c_m^2 \ln m} -1) dx\\
&\geq C (\ln m)^{-\frac p2}(1 + O((\ln m)^{-1}) m^{-1}(e^{c_m^2 \ln m} -1)\\
&\geq C (\ln m)^{-\frac p2}(1 + O((\ln m)^{-1})e^{(c_m^2-1)\ln m},
\end{align*}
here we denote by $C$ for any constant which is independent of $m$. Notice that $c_m^2 = 1 + O((\ln m)^{-1})$ and $\int_{\H^4} w_m^2 dV_g = O((\ln m)^{-1})$. Hence it holds
\[
\frac{1}{\int_{\H^4} w_m^2 dV_g} \int_{\H^4} \frac{e^{32\pi^2 w_m^2} -1}{(1 + w_m^2)^p} dV_g \geq C (\ln m)^{1-\frac p2} (1 + O((\ln m)^{-1})) e^{O(1)},
\]
and
\[
\lim_{m\to \infty}\frac{1}{\int_{\H^4} w_m^2 dV_g} \int_{\H^4} \frac{e^{32\pi^2 w_m^2} -1}{(1 + w_m^2)^p} dV_g \geq C (\ln m)^{1-\frac p2} (1 + O((\ln m)^{-1})) e^{O(1)} = \infty,
\]
if $p < 2$. This proves the sharpness of the power $2$ in the denominator in the integral in \eqref{eq:Adamsexact}.
\end{proof}

\section*{Acknowledgments}
The author would like to thank Quoc Hung Phan for useful discussions in the proof of Theorem \ref{keytool}.


\end{document}